\documentclass{amsart}
\usepackage{amsmath,amsthm}
\usepackage{amsfonts,amssymb}

\usepackage{enumerate}

\newtheorem{thm}{Theorem}[section]
\newtheorem{cor}[thm]{Corollary}
\newtheorem{lem}[thm]{Lemma}
\newtheorem{prop}[thm]{Proposition}

\newtheorem{defn}[thm]{Definition}
\newtheorem{rem}[thm]{Remark}

\def\sph{\mathbb{S}^{d-1}}
\def\f{\frac}
\def\va{\varepsilon}
\def\Bl{\Bigl}
\def\Br{\Bigr}
 \def\a{{\alpha}}
 \def\b{{\beta}}
 \def\g{{\gamma}}
 
 \def\t{{\theta}}
 \def\l{{\lambda}}
 \def\d{{\delta}}
 \def\o{{\omega}}
 \def\s{{\sigma}}
 \def\la{{\langle}}
 \def\ra{{\rangle}}

 \def\CH{{\mathcal H}}

 \def\CO{{\mathcal O}}

 \def\NN{{\mathbb N}}

 \def\RR{{\mathbb R}}
  \def\SS{{\mathbb S}}

        \def\proj{\operatorname{proj}}

        \def\p{\partial}

\newcommand{\wh}{\widehat}

\def\dfrac{\displaystyle\frac}

\def\Bl{\Bigl}
\def\Br{\Bigr}
\def\f{\frac}
\def\({\Bigl(}
\def \){ \Bigr)}

\def\ga{\gamma}
\def\Ga{\Gamma}

\def\HH{\mathcal{H}}

\begin{document}

\title[Hardy-Rellich inequality and uncertainty principle]
{The  Hardy-Rellich inequality and uncertainty principle on the sphere}
\author{Feng Dai}
\address{Department of Mathematical and Statistical Sciences\\
University of Alberta\\, Edmonton, Alberta T6G 2G1, Canada.}
\email{dfeng@math.ualberta.ca}
\author{Yuan Xu}
\address{Department of Mathematics\\ University of Oregon\\
    Eugene, Oregon 97403-1222.}\email{yuan@math.uoregon.edu}
\thanks{The work of the first author was  supported  in part by NSERC  Canada under
grant RGPIN 311678-2010. The work of the second author was supported in part by NSF Grant DMS-1106113
and a grant from the Simons Foundation (\# 209057 to Y. Xu)}

\date{\today}
\keywords{Hardy-Rellich inequality, uncertainty principle, Laplace-Beltrami, spherical gradient, unit sphere}
\subjclass[2000]{42B10, 42C10, 33C45, 33C55, 43A75}

\begin{abstract}
Let $\Delta_0$ be the Laplace-Beltrami operator on the unit sphere $\sph$ of $\RR^d$. We show that the
Hardy-Rellich inequality of the form
$$
  \int_{\mathbb{S}^{d-1}} \left | f (x)\right|^2 d\sigma(x) \leq c_d  \min_{e\in\mathbb{S}^{d-1}}
    \int_{\mathbb{S}^{d-1}} (1- \langle  x, e \rangle) \left |(-\Delta_0)^{\f12}f(x) \right |^2 d\sigma(x)
$$
holds for $d =2$ and $d \ge 4$ but does not hold for $d=3$ with any finite constant, and the optimal
constant for the inequality is $c_d = 8/(d-3)^2$ for $d =2, 4, 5$ and, under additional restrictions on
the function space, for $d\ge 6$. This inequality yields an uncertainty principle
of the form
$$
   \min_{e\in\mathbb{S}^{d-1}}  \int_{\mathbb{S}^{d-1}} (1- \langle x, e \rangle) |f(x)|^2 d\sigma(x)
     \int_{\mathbb{S}^{d-1}}\left |\nabla_0 f(x)\right |^2 d\sigma(x) \ge
      c'_d
$$
on the sphere for functions with zero mean and unit norm, which can be used to establish another
uncertainty principle without zero mean assumption, both of which appear to be new.

This paper is published in Constructive Approximation, 40(2014): 141-171. An erratum is now appended.
\end{abstract}

\maketitle

\section{Introduction}
\setcounter{equation}{0}

The purpose of this paper is to establish an analogue of the  Hardy-Rellich inequality and the
uncertainty principe on the sphere $\sph: = \{ x\in \mathbb{R}^{d}: \|x\|=1\}$, where
$\|x\|$ denotes the Euclidean norm of $x\in \RR^d$. To motivate our results, we first recall
these inequalities on $\RR^d$.

Let $\Delta$ denote the usual Laplace operator on $\RR^d$. For $\a > 0$, $(-\Delta)^{\f \a2}$
denotes the fractional power of $-\Delta$. The inequality of the type
\begin{equation}\label{HardyRellich}
   \int_{\RR^d} |f(x)|^2 \|x\|^{\mu} \, dx \leq c \int_{\RR^d} \left|(-\Delta)^{\f \a2} f(x)\right|^2 \|x\|^{\mu+2\a}\, dx, \end{equation}
is called the Hardy-Rellich-type inequality. It is the classical Hardy inequality when $\a =1$,  and the
Rellich inequality when $\a = 2$. There are many papers devoted to the study of this inequality
and its various generalizations. In particular, the best constant in \eqref{HardyRellich} was calculated in
\cite{Ei, Li, Ya} under some assumptions on the parameters; see also \cite{Sa}. The uncertainty
principle is a fundamental result in quantum mechanics and it can be formulated, in the form of the
classical Heisenberg inequality, as
\begin{equation}\label{UCinequality}
     \inf_{a \in \RR^d} \int_{\RR^d} \|x - a\|^2 |f(x)|^2 dx \int_{\RR^d} |\nabla f(x)|^2 dx
          \ge \frac{d^2}{4} \left(\int_{\RR^d} | f(x)|^2 dx \right)^2.
\end{equation}
The uncertainty principle has been widely studied and extended; see, for example,
\cite{FS, Thanga} and the references therein.

Our main results in this paper are analogues of such results on the unit sphere $\sph$,
in which we work with the Laplace-Beltrami operator $\Delta_0$ and the spherical
gradient $\nabla_0$, which are the restriction of $\Delta$ and $\nabla$
on the sphere, respectively.
Let $d\s(x)$ be  the usual rotation-invariant measure on $\sph$. For smooth functions $f$
on $\sph$ that satisfy $\int_{\sph} f(x) d\s =0$, our main result on  the Hardy-Rellich inequality
states that
\begin{equation} \label{eq:HRsphere}
 \int_{\sph} \left | f (x)\right|^2 d\s(x) \leq c_d  \min_{e\in\sph}
    \int_{\sph} (1- \la  x, e \ra) |(-\Delta_0)^{\f12}f(x)|^2 d\s(x),
\end{equation}
where the constant $c_d$ satisfies $c_d \ge 8/(d-3)^2$, which shows, in particular, a surprising
result that the inequality \eqref{eq:HRsphere} holds for all dimensions but $d =3$, that is,
except for $\SS^2$.  We will also show that the best constant in the inequality is $c_d =8/(d-3)^2$
for all $f$ if $d = 2,4, 5$, and for $f$ in a subspace if $d \ge 6$. We then use the inequality
\eqref{eq:HRsphere} to establish an uncertainty principle, which states that
\begin{equation} \label{eq:UCsphere}
   \min_{e\in\sph}  \int_{\sph} (1- \la x, e \ra) |f(x)|^2 d\s \int_{\sph} |\nabla_0 f(x)|^2 d\s \ge
      c'_d  \left(\int_{\sph}|f(x)|^2 d\s \right)^2
\end{equation}
for smooth functions $f$ satisfying $\int_{\sph} f(x) d\s =0$. The proof, however, is not applicable
for $d = 3$. The gap prompted us to search for a different approach. A second proof shows
that \eqref{eq:UCsphere} does hold for $d =3$.

Recall that the geodesic distance on the sphere is defined by $d(x,y) = \arccos \la x, y\ra$, so that
$$
  1 - \la x, y \ra  = 2\sin^2 \tfrac{d(x,y)}{2},
$$
which shows that \eqref{eq:UCsphere} can be regarded as a close analogue of
\eqref{UCinequality}. Given the numerous extensions of the uncertainty principles on
a wide range of settings, it is somewhat surprising that this formulation of the uncertainty
principle has not appeared, as far as we know, in the literature. The inequality that carries
 the name of the uncertainty principle on the sphere in the literature is (\cite{NW, RV, Selig})
\begin{equation} \label{eq:UCsphere_old}
 \left(1-\|\tau(f)\|^2\right)  \int_{\sph} 
 |\nabla_0 f|^2 d\s  \ge  c \|\tau (f)\|^2
\end{equation}
for smooth functions $f$ satisfying $\|f\|_2 = 1$, where $\tau(f)$ is the vector defined by
$$
\tau (f):=\int_{\sph} x |f(x)|^2  d\s(x).
$$
The inequality \eqref{eq:UCsphere}, however, is stronger than \eqref{eq:UCsphere_old},
since it implies
\begin{equation} \label{eq:UCsphere_new}
 \left(1-\|\tau(f)\|\right) \left ( \int_{\sph} |\nabla_0 f|^2 d\s \right) \ge  c \|\tau (f)\|,
\end{equation}
and we know that $\|\tau(f)\| \le 1$ and $1- \|\tau(f)\| \le 1 -\|\tau (f)\|^2$.
Thus,  our uncertainty principle \eqref{eq:UCsphere} appears to be not
only a close analogue of the classical result on $\RR^d$, but also stronger than what
is known in the literature.

Since the zonal functions $f(\la x,\cdot \ra)$ in $L^2(\sph)$ can be identified with functions
in $L^2(w_\l,[-1,1])$ with $w_\l(t) =(1-t^2)^{\l-1/2}$ and $\l = (d-2)/2$, both the Hardy-Rellich
inequality and the uncertainty principle can be stated for functions in $L^2(w_\l,[-1,1])$ for
$\l = (d-2)/2$, where the operator $\Delta_0$ is replaced by the second order differential
operator that has the Gegenbauer polynomial as the eigenfunctions. Furthermore, these
inequalities can be formulated more generally for all $\l > -1/2$, as we shall do in most of
our statements.

The paper is organized as follows. The next section is devoted to the orthogonal expansions
in spherical harmonics, which will be our main tool. The Hardy-Rellich inequalities are discussed
and proved in Section 3, with the assumption of a technical lemma that will be proved in the
Section 5. The inequalities of uncertainty principle are established in Section 4.

\section{Spherical harmonic expansions}
\setcounter{equation}{0}

Throughout this paper, all functions  are assumed to be real valued and Lebesgue measurable on $\sph$ whenever $d\ge 3$.
Let $L^2(\sph)$ denote the space of functions of finite norm
$$
\|f \|_2 := \left( \f 1{\o_d}
 \int_{\sph} |f(x)|^2 d\s \right)^{1/2} \quad\hbox{with}\quad \o_d:=\f{2\pi^{d/2}}{\Ga(d/2)},
$$
where $\o_d$ is the surface area of the sphere $\sph$ and $d\s(x)/\o_d$ is the normalized
Lebesgue measure on $\sph$.

A spherical polynomial of degree $n$ on $\sph$ is the restriction of an algebraic polynomial
of total degree at most $n$ in $d$-variables on $\sph$. We denote by $ \Pi_n^d$ the space of
real spherical polynomials of degree at most $n$ on $\sph$. A spherical harmonic of degree
$n$ in $d$-variables is the restriction of a homogeneous harmonic polynomial of degree $n$
on $\sph$. We denote by $\CH_n^d$, $n=0,1,\cdots$, the space of spherical harmonics of
degree $n$ on $\sph$, which has dimension
\begin{equation}
 a_n^d: = \dim \CH_n^d = \dfrac{(2n+d-2)\Gamma(n+d-1)}{(n+d-2)\Gamma(n+1)\Gamma(d-1)},
  \quad  n=0,1,\cdots.
\end{equation}
These spaces are known to be mutually orthogonal with respect to the inner product of
$L^2(\sph)$. Since the space of spherical polynomials is dense in $L^2(\sph)$,  we have the
orthogonal decomposition
\begin{equation}\label{s-expansions}
L^2(\sph)=\bigoplus_{n=0}^\infty \mathcal{H}_n^d: \qquad  f=\sum_{n=0}^\infty \proj_n f,
\end{equation}
where $\proj_n$ is the orthogonal projection of $L^2(\sph)$ onto the space $\HH_n^d$.

The restriction of the Laplace operator on the the sphere is the Laplace-Beltrami operator
$\Delta_0$, which is defined by
$$
  \Delta_0 f := \Delta F \bigl|_{\sph}, \quad \hbox{where}\,\, F(x) = f\left(\f x{\|x\|}\right).
$$
For each $n =0,1,\ldots,$ the space of spherical harmonics $\HH^d_n$ is the eigenfunction-space
of $\Delta_0$ with the eigenvalue $-n(n+d-2)$, that is,
\begin{equation*}
   \HH^d_n =\left \{ f\in C^2(\SS^{d-1}): \Delta_0 f =- n (n+d-2)f\right\}, \quad n=0,1,\cdots.
\end{equation*}
For $r\in\RR\setminus\{0\}$, the fractional Laplace-Beltrami operator $(-\Delta_0)^{r}$ is
defined in a distributional  sense through $\proj_0 \left[(-\Delta_0)^r f \right]=0$ and
\begin{equation}\label{1:ch0}
   \proj_n\left[(-\Delta_0)^r f \right]= (n (n+d-2))^{r} \proj_n(f), \quad n=1, 2, \cdots.
\end{equation}
Let $\nabla$ denote the usual gradient operator of $\RR^d$. Then the tangential gradient
$\nabla_0 f$ of a function $f\in C^1(\sph)$ is defined by
$$
   \nabla_0 f =\nabla F \bigl|_{\sph}, \quad \hbox{where} \,\, F(x) = f\left(\f x{\|x\|}\right).
$$
It is known (\cite[p.80, Lemma 1]{Mu}) that, for $f, g \in C^2(\sph)$,
\begin{equation*}
   \la \Delta_0 f, g\ra_{L^2(\sph)} =-\int_{\sph} \la \nabla_0 f,  \nabla_0 g \ra d\s(x),
\end{equation*}
which, in particular,  implies, since $\Delta_0$ is self-adjoint in $L^2(\sph)$, that
\begin{equation}\label{1-6}
 \left \|(-\Delta_0)^{1/2} f \right \|_2 =  \left \| \la \nabla_0 f, \nabla_0 f \ra^{1/2} \right\|_2
       =: \|\nabla_0 f\|_2.
\end{equation}

When $d =2$, we parametrize $\SS^1$ by $x = e^{i\t}$ for $\t \in [0,2\pi)$ and identify
$f(e^{i\t})$ with $f(\t)$. Choosing $\{e^{in\t}, e^{- in\t}\}$ as a basis of $\CH^2$, the
function $f \in L^2(\SS^1)$ has the usual Fourier series
\begin{equation} \label{FourierSeries}
    f(\t) = \sum_{n=-\infty}^\infty \wh f_n e^{i n\t}, \quad \hbox{where}\quad \wh f_n
             = \frac {1}{2\pi} \int_{0}^{2\pi} f(t) e^{-i nt} dt.
\end{equation}
In this case $\proj_n f = \wh f_n e^{in\t} + \wh f_{-n} e^{- i n\t}$, $\nabla_0 = \frac{d}{d\t}$
and $\Delta_0 = \f{d^2}{d\t^2}$.

For $d >2$, we will need an explicit form of an orthonormal basis for $\CH_n^d$
parametrized by $x = (\cos \t, \sin \t \, \xi) \in  \sph$, where $\xi \in \SS^{d-2}$ and
$0 \le \t \le \pi$. This basis can be derived from the usual basis in spherical coordinates;
see, for example, \cite[p. 35]{DX}. For completeness, we give an independent derivation
below. For $\l >  -1/2$ and $n \in \NN_0$, let $C_n^\l$ denote the Gegenbauer polynomial of
degree $n$. The polynomials $C_n^\l$ satisfy the orthogonal relation  \cite[(4.7.15)]{Sz}
\begin{equation} \label{Gegen}
   c_\l \int_{-1}^1 C_n^\l(t) C_m^\l(t) (1-t^2)^{\l-1/2} dt  = h_n^\l \d_{m,n}, \quad
        h_n^\l: = \frac{ \l (2 \l)_n}{(n+\l)n!},
\end{equation}
where $(a)_n = a(a+1)\cdots (a+n-1)$ is the Pochhammer symbol and $c_\l$ is the
normalization constant $c_\l =  1 \Big /\int_{-1}^1(1-t^2)^{\l-1/2} dt =
\f{\Ga(\l+1)}{\sqrt{\pi} \Ga(\l+\f12)}$.

\begin{prop}\label{prop-2-3}
Let $\l = \f{d-2}{2}$ and $d > 2$. For $m \in \NN_0$, let $\{Y_j^m(\xi): 1 \le j \le a_m^{d-1}\}$ be
an orthonormal basis of $\CH_m^{d-1}$. For  $x = (\cos \t, \xi\sin \t )\in\sph$ with
$0\le \t \le \pi$ and $\xi \in \SS^{d-2}$, we define
$$
  P_{j,k}^n(x) = C_k^{n-k+\l}(\cos \t) (\sin\t)^{n-k} Y_j^{n-k}(\xi), \quad  1 \le j \le a_{n-k}^{d-1}, \,
   0 \le k \le n.
$$
Then $\{P_{j,k}^n: 1 \le j \le a_{n-k}^{d-1}, \, 0 \le k \le n\}$ is an orthogonal
basis of $\CH_n^d$ and
\begin{align} \label{Hkn}%\label{1-5}
 H_{k}^n : =  \frac{1}{\o_d} \int_{\sph}\left[P_{j,k}^n(x) \right]^2 d\s(x)
       = h_k^{n-k+\l}.
%      = &   \f{\Ga(\l+1)}{n+\l}\f{ \Ga(2n-k+2\l)\Ga(n-k+\l+\f12)}{k! \Ga(n-k+\l)},
\end{align}
\end{prop}

\begin{proof}
Using  the integral formula
$$
  \f 1 {\o_d} \int_{\sph} f(x) d\s_d(x) = c_\l \int_{0}^\pi  \left[\f{1}{\o_{d-1}} \int_{\SS^{d-2}}
     f(\cos \t, \xi\sin \t ) d\s_{d-1}(\xi) \right] (\sin\t)^{2\l} d\t,
$$
and the orthonormality of $Y_j^{n-k}$, we obtain that
\begin{align*}
\la P_{j,k}^n, P_{j',k'}^{n}\ra_{L^2(\sph)} & = c_\l\d_{j,j'} \d_{n-k, n-k'} \\
   & \times  \int_{0}^{\pi} C_k^{n-k+\l}(\cos \t)C_{k'}^{n-k'+\l}(\cos \t) (\sin\t)^{2n-k-k'+2\l} d\t,
\end{align*}
from which the mutual orthogonality of  $P_{j,k}^n$ follows, so is the formula of $H_k^n$.

Since each $Y_j^{n-k}$ is the restriction to $\SS^{d-2}$ of a homogeneous polynomial
in $d-1$ variables of degree $n-k$, it follows readily that, for $x=(x_1,\cdots, x_d)\in\sph$,
$$
  P_{j,k}^n (x) = C_k^{n-k+\l}(x_1) Y_{j}^{n-k}(x_2, \cdots, x_d),
$$
which shows that  $P_{j,k}^n$ is a homogeneous polynomial. Furthermore, it
is easy to verify that $ \sum_{k=0}^n a_{n-k}^{d-1} = a_n^d 
= \text{dim}\, \CH_n^d$. Since the orthogonality determines the spherical harmonics, $\{P_{j,k}^n:
1 \le j \le a_{n-k}^{d-1}, \, 0 \le k \le n\}$ is a basis of $\CH_n^d$.
\end{proof}

\begin{defn}
For $d > 2$, we define the Fourier coefficients of  $f\in L^2(\sph)$ with respect to the mutually
orthogonal basis $\{P_{j,k}^n(x)\}$ by
\begin{equation}\label{2-3}
    \wh f_{j,k}^n : = [H_{k}^n]^{-1/2} \f 1 {\o_d} \int_{\sph} f(y) P_{j,k}^n(y) d\s(y),\   \  0\leq k\leq n,\   \ 1\leq j\leq a_{n-k}^{d-1}.
\end{equation}
\end{defn}

As a direct consequence of Proposition \ref{prop-2-3}, the projection operator can be
expressed as the following:

\begin{lem}\label{cor-2-3}
For each $f\in L^2(\sph)$, $d > 2$, and  $n\in\NN_0$,
\begin{equation}\label{2-4}
   \proj_n f(x) = \sum_{k=0}^n \sum_{j=1}^{a_{n-k}^{d-1}}\wh f_{j,k}^n  [H_{k}^n]^{-1/2} P_{j,k}^n(x),\end{equation}
   and \begin{equation}\label{2-5}
 \|\proj_n f\|_2^2  = \sum_{k=0}^n   \sum_{1 \le j \le a_{n-k}^{d-1}} \left| \wh f_{j,k}^n \right |^2.
\end{equation}
\end{lem}

The reason for our choice of the particular basis in Proposition \ref{prop-2-3} lies in the following
result.

\begin{lem} \label{lem-3-2}
Let $d > 2$. If $f\in L^2(\sph)$ and $\int_{\sph} f(x)\, d\s(x)=0$, then
\begin{align}
 \f1 {\o_d} \int_{\sph} x_1 |f(x)|^2 d\s(x)     & =    \sum_{n=1}^\infty
    \sum_{k=0}^n \g_k^n \sum_{1 \le j \le a_{n-k}^{d-1}} \wh f_{j,k}^n\wh f_{j,k+1}^{n+1},\label{3-4}
\end{align}
where \begin{equation}\label{3-5}
  \g_k^n:=  \sqrt{\frac{(2n-k+2\l)(k+1)}{(n+\l)(n+\l+1)}}.
\end{equation}
\end{lem}

\begin{proof}
Firstly, we note that, by
the three term relation of the Gegenbauer polynomials (see \cite[p.81, (4.7.17)]{Sz}), for $x=(\cos\t, \xi\sin\t)$ with  $\xi\in\SS^{d-2}$ and $\t\in [0,\pi]$,
\begin{align*}
   x_1 P_{j,k}^n(x)& = \left[A_k^n C_{k+1}^{n-k+\l}(\cos \t)+ B_k^nC_{k-1}^{n-k+\l}(\cos \t) \right] (\sin\t)^{n-k} Y_j^{n-k}(\xi) \\
   & = A_k^n P_{j,k+1}^{n+1}(x) + B_k^n P_{j,k-1}^{n-1}(x),
\end{align*}
where   the coefficients are given by
$$
  A_k^n: = \frac{k+1}{2(n+\l)} \quad \hbox{and}\quad B_k^n: = \frac{2n-k+2\l-1}{2(n+\l)},
$$
and we assume  that $P_{j,-1}^{n-1}(x)=0$.
In particular, this implies
$$
  x_1 \proj_n f(x) = \sum_{k=0}^n \sum_{1 \le j \le a_{n-k}^{d-1}} [H_{k}^n]^{-\f12}\wh f_{j,k}^n
      \left[A_k^n P_{j,k+1}^{n+1}(x) + B_k^n P_{j,k-1}^{n-1}(x) \right].
$$
Consequently, by the orthogonality of $P_{j,k}^n$, it follows that
$$
   \int_{\sph} x_1 \proj_n f(x) \proj_m f(x) d\s = 0, \quad \hbox{unless $|m-n|=1$},
$$
and that
\begin{align*}
  &\f1 {\o_d} \int_{\sph} x_1 \proj_n f(x) \proj_{n+1} f(x) d\s\\
  & =
    \sum_{k=0}^n A_k^n [H_{k+1}^{n+1}]^{-\f12} [H_k^n]^{-\f12} \sum_{1 \le j \le a_{n-k}^{d-1}} \wh f_{j,k}^n\wh f_{j,k+1}^{n+1}.
\end{align*}
Consequently, we obtain that
\begin{align*}
 \f1 {\o_d} \int_{\sph} x_1 |f(x)|^2 d\s  & = 2  \sum_{n=1}^\infty
     \f1 {\o_d} \int_{\sph} x_1  \proj_n f(x)  \proj_{n+1} f(x) d\s(x)\\
      & =    \sum_{n=1}^\infty
    \sum_{k=0}^n \g_k^n \sum_{1 \le j \le a_{n-k}^{d-1}} \wh f_{j,k}^n\wh f_{j,k+1}^{n+1},
\end{align*}
where the first step uses the assumption that  $\proj_0 f(x)  = \f 1 {\o_d} \int_{\sph}f(x) d\s =0$. This completes the proof.
\end{proof}

A zonal function on the sphere is a function that depends only on $\la x, y\ra$, that is, a function
of the form $f_0(\la x,y\ra)$.
It is well known that the reproducing kernel $P_n(\cdot,\cdot)$ of $\CH_n^d$ in $L^2(\sph)$ is given by
a zonal polynomial
$$
   Z_n(x, y) = \frac{n+\l}{\l} C_n^\l(\la x, y\ra), \quad \l = \f{d-2}{2},
$$
which is the integral kernel of $\proj_n f$, that is,
$$
      \proj_n f(x)  = \f{1}{\o_d} \int_{\sph} f(y) Z_n(x,y) d\s (y), \quad \forall x \in \sph.
$$
For a function $f$ defined on $[-1,1]$, it is well known that the spherical harmonic expansion of a
zonal function $x \mapsto f(\la x,y \ra)$ agrees with the Gegenbauer expansion of $f$ in
$C_n^\l$ with $\l = \f{d-2}{2}$.

The connection to the Gegenbauer expansions holds for general parameters of $\l$. For
$f \in L^2(w_\l,[-1,1])$ with $w_\l(t) = (1-t^2)^{\l-1/2}$, the Gegenbauer expansion of $f$ is
given by
$$
    f(t) = \sum_{n=0}^\infty \wh f_{n}^\l (h_n^\l)^{-\frac12}C_n^\l(t), \qquad
       \wh f_n^\l :=  (h_n^\l)^{-\frac12}c_\l \int_{-1}^1 f(s) C_n^\l(s) w_\l(s) ds,
$$
where $c_\l$ denotes the normalization constant of $w_\l$, which follows from the fact
that $(h_n^\l)^{-\frac12}C_n^\l(t)$ is orthonormal and the identity holds in the $L^2$ sense.
As in the proof of Lemma \ref{lem-3-2}, we can deduce from the three-term relation of
the Gegenbauer polynomials the following result:

\begin{prop}\label{prop:s-GegenSeries}
For $\l >  -1/2$ and $f \in L^2(w_\l,[-1,1])$,
\begin{equation}\label{eq:s-GegenSeries}
   c_\l \int_{-1}^1 s |f(s)|^2 w_\l(s) ds = \sum_{n=1}^\infty \g_n^n \wh f_n^\l \wh f_{n+1}^\l.
\end{equation}
\end{prop}

For $\l = 0$, the Gegenbauer polynomials become the Chebyshev polynomials of the first
kind, or the cosine functions upon setting $t = \cos \t$, which correspond to the zonal
functions in the case of $\SS^1$. For the Fourier series in \eqref{FourierSeries}, we have
\begin{equation}\label{eq:s-FourierSeries}
  \frac{1}{2\pi} \int_{0}^{2 \pi} (\cos \t)  |f(\t)|^2 d\t  = \sum_{n=-\infty}^\infty \wh f_n \wh f_{n+1},
\end{equation}
which can be easily verified upon using $\cos \t = (e^{i\t} + e^{-i \t})/2$.

\section{The  Hardy-Rellich-type inequality}
\setcounter{equation}{0}

Let us start with the simple case of $\SS^1$, the proof of which nevertheless indicates what
is needed in the higher dimension. What we need is an inequality that can be deduced
from the classical Hardy inequality. The Hardy inequality (cf. \cite[p. 239, (9.8.1)]{HLP}) states
that for $1<p<\infty$ and any sequence of real numbers $b_n$,
\begin{equation} \label{eq:Hardy}
   \sum_{n=1}^\infty  \left( \frac 1 n \sum_{k=1}^n b_k \right)^p \le \left (\f p {p-1} \right)^p
     \sum_{n=1}^\infty |b_n|^p.
\end{equation}

\begin{lem}\label{lem-3-1}
If $\{ a_k\}_{k=1}^\infty$ is a sequence  of real numbers, then
\begin{equation}\label{3-6}
\sum_{n=1}^\infty |a_n a_{n+1}|\leq \sum_{n=1}^\infty \left(1-\f 1{8n^2}\right) a_n^2.
\end{equation}
\end{lem}

\begin{proof}
Without loss of generality, we may assume that  $a_n\ge 0$ for all $n\in\NN$, and that $\sum_{n=1}^\infty a_n^2<\infty$.
Setting $a_0=0$ and $b_n =a_n-a_{n-1}$ for $n\ge 1$, we can rewrite \eqref{eq:Hardy}
in the following equivalent form:
\begin{equation*}
   \sum_{n=1}^\infty n^{-p} a_n^p \leq \Bigl(\f p{p-1}\Bigr)^p \sum_{n=1}^\infty (a_n-a_{n-1})^p,
\end{equation*}
which, upon setting $p=2$ and using $(a_n-a_{n-1})^2=a_n^2+a_{n-1}^2 -2a_n a_{n-1}$, can
be rearranged to give the desired inequality \eqref{3-6}.
\end{proof}

Recall that for $f$ defined on $\SS^1$, we identify $f(e^{i\t})$ with $f(\t)$ for $\t \in [0, 2\pi)$.
The Hardy-Relich inequality on $\SS^1$ takes the following form:

\begin{thm} \label{thm:HRineq-d=2}
Let $f \in L^2(\SS^1)$ satisfy $f' \in L^2(\SS^1)$ and $\int_{0}^{2\pi}  f(\t) d\t =0$. Then
\begin{equation}\label{eq:HRineq-d=2}
     \int_{0}^{2\pi} (1- \cos \t) |f'(\t)|^2 d\t  \ge  \frac{1}{8} \int_{0}^{2\pi} |f(\t)|^2 d\t.
\end{equation}
Furthermore, the constant $1/8$ is sharp.
\end{thm}

\begin{proof}
The assumption implies that $\wh f_0 =0$. Applying the inequality \eqref{3-6} to
\eqref{eq:s-FourierSeries} shows that
$$
  \sum_{n= -\infty}^{\infty} |\wh f_n \wh f_{n+1} |
      = \sum_{n=1}^\infty |\wh f_n \wh f_{n+1} | + \sum_{n=1}^\infty |\wh f_{-n} \wh f_{-n+1} |
   \le \sum_{\substack{n=-\infty \\ n \ne 0}}^\infty \left(1- \frac{1}{8 n^2}\right) |\wh f_n|^2,
$$
which implies, by the Parseval identity and \eqref{eq:s-FourierSeries}, that
\begin{align*}
   \frac{1}{2\pi} \int_0^{2\pi} (1-\cos \t) |f(\t)|^2 d \t   = \sum_{n=-\infty}^\infty |\wh f_n|^2 -
      \sum_{n=-\infty}^\infty  \wh f_n \wh f_{n+1}
   \ge   \frac{1}{8} \sum_{\substack{n=-\infty \\ n \ne 0}}^\infty \frac{1}{n^2} |\wh f_n|^2.
\end{align*}
Applying the above inequality with $f$ replaced by $f'$, the stated result follows from
the fact that $\wh f'_n = n \wh f_n$ and the Parseval identity. That the constant $1/8$
is sharp is proved later in Theorem \ref{thm:HR-iff}.
\end{proof}

We note that the condition $\int_0^{2\pi} f(\t) d\t = 0$ is necessary for the inequality
\eqref{eq:HRineq-d=2}, as it can be seen by setting $f(\t) = 1$. Such a condition is
also necessary for the Hardy-Rellich inequality on $\sph$ for $d > 2$.

For $d > 2$ and $\a>0$, we define the Sobolev space $W_2^\a$ on $\sph$ by
$$
  W_2^\a:=\left \{ f\in L^2(\sph):  (-\Delta_0)^{\a/2} f\in L^2(\sph)\right \}.
$$

\begin{thm}\label{thm:main-HR}
 If $d\ge 4$, $f\in W_2^1(\sph)$ and $\int_{\sph} f(x) d\s(x)=0$,  then
\begin{equation} \label{eq:main-HR}
 \int_{\sph} \left | f (x)\right|^2 d\s(x)\leq c_d  \min_{e\in\sph}
    \int_{\sph} (1- \la x,  e\ra ) |(-\Delta_0)^{\f12}f(x)|^2 d\s(x),
\end{equation}
where the positive constant $c_d$ depends only on $d$.
\end{thm}

\begin{proof}
By rotation invariance of the Lebesgue measure $d\s(x)$, without loss of generality, we may assume
that $e=(1,0,\cdots,0)$. Let
\begin{equation} \label{eq:J(f)}
   J(f): = \f1 {\o_d} \int_{\sph} (1-x_1)\left |(-\Delta_0)^{\f12} f(x)\right |^2 d\s.
\end{equation}
Using Lemma \ref{lem-3-2},
\begin{align*}
 \f1 {\o_d} \int_{\sph} x_1 \left |(-\Delta_0)^{\f12} f(x)\right |^2 d\s(x)  =  \sum_{n=1}^\infty
    \sum_{k=0}^n \g_k^n \sum_{1 \le j \le a_{n-k}^{d-1}} \wh g_{j,k}^n\wh g_{j,k+1}^{n+1},
\end{align*}
where $\wh{g}_{j,k}^n=\sqrt{n(n+2\l)}\wh{f}_{j,k}^n$. The constants $\g_k^n$ for $1 \le k \le n$
and $\g_n^n$ can be rewritten  as follows:
$$
\ga_k^n =\sqrt{ \f{(n+\l+\f12)^2 -( n+\l-\f12 -k)^2 }{ (n+\l)(n+\l+1)}} \quad \hbox{and}\quad
   \g_n^n = \sqrt{1 - \frac{\l(\l-1)}{(n+\l)(n+\l+1)}},
$$
which shows that $\g_k^n$ is an increasing function in $k$ and $\g_n^n$ is an increasing
function in $n$ if $\l \ge 1$, or equivalently,  $d \ge 4$. Using these facts and
$2|\wh g_{j,k}^n\wh g_{j,k+1}^{n+1}| \le |\wh g_{j,k}^n|^2 + |\wh g_{j,k+1}^{n+1}|^2$, we
conclude that for $d \ge 4$,
\begin{align*}
 & \f1 {\o_d} \int_{\sph}  x_1 \left |(-\Delta_0)^{\f12} f(x)\right |^2 d\s   \le  \frac12
    \sum_{n=1}^\infty   \g_n^n  \sum_{k=0}^n  \sum_{j=1}^{a_{n-k}^{d-1}}
          \left( |\wh g_{j,k}^n|^2  + |\wh g_{j,k+1}^{n+1}|^2 \right).
\end{align*}
Consequently, we deduce easily that
\begin{align*}% \label{eq:J(f)}
   J(f) \ge \sum_{n=1}^\infty  \sum_{k=0}^n (1-\g_n^n)\sum_{j=1}^{a_{n-k}^{d-1}} |\wh g_{j,k}^n|^2.
\end{align*}
It follows from the expression
\begin{align}\label{eq:1-gamma_n}
 1- \g_n^n = \frac{1}{1+\sqrt{\frac{(n+2\l)(n+1)}{(n+\l)(n+\l+1)}} } \frac{ (\l-1)\l}{(n+\l)(n+\l+1)}
\end{align}
that $(1 -\g_n^n) n (n+\l)$ is bounded bellow by a constant $c >0$ for $\l  >1$. Consequently,
if $d > 4$ then
$$
   J(f) \ge c \sum_{n=1}^\infty \sum_{k=0}^n \sum_{1 \le j \le a_{n-k}^{d-1}} |\wh f_{j,k}^n|^2
       = c \|f\|_2^2.
$$
If $d =4$, then $\l =1$ and $\g_n^n =1$, we use $\g_k^n \le \g_n^n =1$ and the
Cauchy-Schwartz inequality, followed by Lemma \ref{lem-3-1}, to conclude that
\begin{align}\label{eq:key-d=4}
 \f1 {\o_d} \int_{\sph} x_1 \left |(-\Delta_0)^{\f12} f(x)\right |^2 d\s & \leq \sum_{n=1}^\infty
    \Biggl ( \sum_{k=0}^{n} \sum_{j=1}^{a_{n-k}^{d-1}} |\wh{g}_{j,k}^n|^2\Biggr)^{\f12}
    \Biggl( \sum_{k=0}^{n} \sum_{j=1}^{a_{n-k}^{d-1}} |\wh{g}_{j,k+1}^{n+1}|^2\Biggr)^{\f12} \notag \\
     & \le \sum_{n=1}^\infty \left (1-\f 1{8n^2}\right )\sum_{k=0}^{n} \sum_{j=1}^{a_{n-k}^{d-1}}
      |\wh{g}_{j,k}^n|^2,
\end{align}
which implies immediately that
$$
  J(f) \ge  \frac{1}{8} \sum_{n=1}^\infty \f1{n^2} \sum_{k=0}^{n} \sum_{j=1}^{a_{n-k}^{d-1}}
      |\wh{g}_{j,k}^n|^2 \ge \frac{1}{8} \sum_{n=1}^\infty  \sum_{k=0}^{n} \sum_{j=1}^{a_{n-k}^{d-1}}
      |\wh{f}_{j,k}^n|^2 = \frac{1}8 \|f\|_2^2,
$$
by the definition of $\wh g_{j,k}^n$ and the Parseval identity.
\end{proof}

The above proof does not produce an optimal constant for the inequality for $d > 4$, although
we can deduce explicit expression for the constant from the proof. The case $d = 4$ is more
delicate than the case $d > 4$, as it requires the Hardy inequality, just as the case of $d =2$
in Theorem \ref{thm:HRineq-d=2}. The case $d =3$ is left open in the above two theorems.
In the following we will address the problem of optimal constant, which also answers the
question on $d =3$. The key step lies in the case of $L^2(w_\l, [-1,1])$, which corresponds to
the zonal functions in $L^2(\sph)$ when $\l = \f{d-2}{2}$, which we consider first.

For $\l > -1/2$, the norm of the space $L^2(w_\l, [-1,1])$ is defined by
$$
  \|f\|_{\l, 2} : = \left( c_\l \int_{-1}^1 |f(t)|^2 w_\l(t) dt \right)^{1/2}.
$$
The differential operator that has the Gegenbauer polynomials as eigenfunctions is defined by
$$
   D_\l := (1-t^2) \frac{d^2}{dt^2} - (2\l+1) \frac{d}{dt},
$$
which is the restriction of $\Delta_0$ on functions of the form $f(x) = f(x_1)$ with
$x = (x_1,\ldots,x_d) \in \sph$, and
$$
  D_\l C_n^{\l}(t) = - n (n+2\l) C_n^\l(t), \qquad n =0,1,2, \ldots.
$$
Let us also define, for $\a \in \RR$,
$$
   W_2^\a (w_\l, [-1,1]) : =\left \{ f\in L^2(w_\l,[-1,1]): (-D_\l)^\a f\in L^2(w_\l, [-1,1]) \right\}.
$$
We start with the following theorem.

\begin{thm} \label{thm:HR-Gegen-first}
For $\l > -1/2$, let $f \in L^2(w_\l,[-1,1]) \cap W_2^1(w_\l,[-1,1])$ satisfy
$\int_{-1}^1 f(t) w_\l(t) dt =0$. If $\l \neq \f12$, then
\begin{equation} \label{eq:HR-Gegen}
       \int_{-1}^1 |f(t)|^2 w_\l(t) dt  \le  C_\l \int_{-1}^1 (1-t) \left | (-D_\l)^{\f12} f(t)\right|^2 w_\l(t) dt,
\end{equation}
where $C_\l$ is a positive constant depending only on $\l$,
and in the case when $0\leq \l\leq 1$ and $\l\neq \f12$,
$
   C_\l = \frac{8}{(2 \l - 1)^2},
$
and it is optimal. The inequality \eqref{eq:HR-Gegen} fails when $\l=\f12$.
\end{thm}

\begin{proof}
First, we prove the result for  the cases of  $\l>1$ and  $-\f 12 <\l <0$,
where the optimal constant is not known and hence the proof is much  easier.
Let $\wh g_n^\l = \sqrt{n(n+2 \l)}\wh f_n^\l$. Using \eqref{eq:s-GegenSeries} and $2ab =a^2+b^2 -(a-b)^2$, we obtain
\begin{align*}
  c_\l& \int_{-1}^1 s \left | (-D_\l)^{\f12} f(s)
  \right|^2 w_\l(s) ds  = \f12
      \sum_{n=1}^\infty \g_n^n \left(|\wh g_n^\l|^2+ |\wh
      g_{n+1}^\l|^2 -  |\wh g^\l_n - \wh g^\l_{n+1}|^2\right) \\
     & =   \sum_{n=1}^\infty \f{\g_{n-1}^{n-1}+\g_n^n}2 |\wh g^\l_n|^2 - \f12 \g_0^0 |\wh g_1^\l|^2 -  \f 12 \sum_{n=1}^\infty \g_n^n
          |\wh g^\l_n - \wh g^\l_{n+1}|^2\leq \sum_{n=1}^\infty \g_n^n |\wh g^\l_n|^2,
\end{align*}
where the last step uses the fact that $\ga_n^n$ is nonnegative and increasing in $n$ when $\l(\l-1)>0$.
This implies that
\begin{align} \label{J(f)-Gegen}
  J_\l(f): = & c_\l \int_{-1}^1 (1-t) \left| (-D_\l)^{\f12} f(t)\right|^2 w_\l(t) dt \notag\\
  &\ge \sum_{n=1}^\infty (1- \g_n^n) |\wh g^\l_n|^2
  = \sum_{n=1}^\infty  \gamma_\l (n) |\wh f^\l_n|^2,
  \end{align}
  where
 $\g_\l(n): =(1-\ga_n^n)n (n+2\l).$
 Using  \eqref{eq:1-gamma_n}, we may write 
\begin{align*}
  \g_\l(n) =\frac{\l(\l-1) x_n }{1+\sqrt{x_n} }  \frac{ n}{n+1},
\end{align*}
with
$$x_n:= \frac{(n+2\l)(n+1)}{(n+\l)(n+\l+1)}=1-\f{\l(\l-1)}{(n+\l)(n+\l+1)}.$$
  Note that 
  $x_n$ is an increasing function  in $n$ when $\l(\l-1)>0$. Since $x/ (1+\sqrt{x})$ is an increasing
function for $x > 0$,  it follows that
\begin{align*}
  \g_\l(n)\ge \f12\frac{\l(\l-1) x_1 }{1+\sqrt{x_1} }=:C_\l>0.
\end{align*}
This together with \eqref{J(f)-Gegen}
implies the desired estimate \eqref{eq:HR-Gegen} in the case when $\l>1$ or $-\f 12 <\l<0$.

  Next, we prove 
   the  estimate \eqref{eq:HR-Gegen} with the optimal constant $C_\l:= \f 8{(2\l-1)^2}$
   for  $\l\in [0,1]$ and $\l\neq \f12$.   
   The proof
is quite  involved. It relies on  an observation that $\g_n^n$ admits a factorization in the
form of $\a_n^\l \a_{n+1}^\l$; namely, $\ga_n^n:=\a_n^\l \a_{n+1}^\l$,  where
\begin{align}
\a_{2n+1}^\l :&=\sqrt{  \f {2\Ga (n+\f 32) \Ga(n+1+\l)}{(2n+\l+1)\Ga(n+1) \Ga(n+\l+\f12)}},\    \   n=0,1,\cdots, \label{3-3-0}\\
\a_{2n}^\l:&=\sqrt{  \f{2 n! \Ga(n+\l+\f12)}{(2n+\l) \Ga(n+\f12) \Ga(n+\l)}},\   \  n=1,2,\cdots.\label{3-4-0}
\end{align}
 Using \eqref{eq:s-GegenSeries}, we have
\begin{align*}
  c_\l \int_{-1}^1 s \left | (-D_\l)^{\f12} f(s)
  \right|^2 w_\l(s) ds & =
      \sum_{n=1}^\infty \g_n^n \wh g_n^\l\wh
      g_{n+1}^\l =\sum_{n=1}^\infty \a_n^\l \a_{n+1}^\l  \wh g_n^\l\wh
      g_{n+1}^\l\\
&\leq \sum_{n=1}^\infty \Bl( 1-\f 1{8n^2}\Br) |\a^\l_n|^2|\wh g_n^\l|^2.
\end{align*}
Let us define, for $n \in \NN$, and $\a_n:= \a_n^\l$,
\begin{equation}
  \b_\l(n): = \left(1- \alpha_n^2 -\f{\alpha_n^2}{8n^2} \right) n(n+2\l).
\end{equation}
It follows that 
\begin{align*}
   &\sum_{n=1}^\infty |\wh g_n^\l|^2 -\sum_{n=1}^\infty \ga_n^n \wh g_n^\l \wh g_{n+1}^\l\ge
    \sum_{n=1}^\infty\Bl(1- \a_n^2 +\f {\a_n^2}{8n^2}\Br)|\wh g_n^\l|^2 \\
    &=\sum_{n=1}^\infty\b_\l(n) |\wh f_n^\l|^2\ge \Bl(\inf_{n\ge 1}\b_\l(n) \Br) \|f\|_{2,\l}^2.
\end{align*}
However, by Lemma \ref{lem:key-lemma} below,
$$\inf_{n\ge 1}\b_\l(n) =\b_\l(\infty):= \f {(2\l-1)^2}8,\   \  \l\in [0,1].$$
This completes the proof of \eqref{eq:HR-Gegen} for the case of $\l\in [0,1]$.

Finally,  we point out that the optimality of the constant $C_\l:= \f 8{(2\l-1)^2}$ and the fact that \eqref{eq:HR-Gegen} fails for  $\l=\f12$  are contained in Theorem \ref{thm:HR-iff} below.
\end{proof}

For convenience, we define $n(\l)$ to be the smallest positive integer such that
$$
   \min \{ \beta_\l(n) : n \ge n(\l)\} = \b_\l (\infty).
$$

\begin{lem} \label{lem:key-lemma}
The following statements hold:
\begin{enumerate}[\rm (i)]
\item $\g_n^n = \a_n^\l \a_{n+1}^\l$ for all $n\in\NN$.
\item The  sequences $\{\a_{2n+1}^\l \}_{n=0}^\infty$ and $\{\a_{2n}^\l \}_{n=1}^\infty$ are
   decreasing when $0\leq  \l\leq  1$ and increasing when $\l > 1$ or $\l < 0$.
\item $\lim_{n\to\infty}\b_\l(n)=\b_\l(\infty): = (2 \l -1)^2/8$.
\item For $n \ge 3 \l^3/2$, $\{\b_\l(2n+1)\}_{n=n_0}^\infty$ and $\{\b_\l(2n)\}_{n=n_0}^\infty$
 both decrease to $\b_\l(\infty)$;  in particular, $n(\l) \le 3 \l^3/2$.
\item $n(1/2) = n(1) = n(2/3)=0$, and $n(2)= 4$.
\end{enumerate}
\end{lem}

The proof of this  lemma quite technical and therefore  is delayed till the appendix.

For convenience, we set, for a given integer $k\in \NN$,
$$
  L_k^2(w_\l,[-1,1]):=\Bigl\{ f\in L^2(w_\l,[-1,1]):
  \wh f_j^\l = 0, \quad 0\le j \le k\Bigr\}.
$$

\begin{thm}\label{thm:HR-iff}
If for some $n_0 \in \NN_0$ the inequality
\begin{align} \label{eq:HR-Gegen2}
\int_{-1}^1 |f(t)|^2 w_\l(t) dt \leq C \int_{-1}^1 (1-t) \left |(-D_\l)^{\f12} f(t)\right|^2 w_\l(t) dt
\end{align}
holds for all $f\in L_{n_0}^2 \cap W_2^1(w_\l,[-1,1])$, then
\begin{equation}\label{3-14}
     C \ge  C_\l: = \f 8{(2\l-1)^2}.
\end{equation}
In particular, the inequality \eqref{eq:HR-Gegen2} does not hold with a finite
constant if $\l = 1/2$. Furthermore, the equality $C= C_\l$ is attained if $n_0 = n(\l)$.
\end{thm}

\begin{proof}
Assume that \eqref{3-14} were not true, then there would be an  $\va > 0$ such that
$$
    C^{-1}-\va > \f {(2\l -1)^2}8 = \lim_{n\to\infty} \beta_\l(n),
$$
which implies that  there exists a positive integer $N_0 > n_0$ such that
$$
  \beta_\l(n)=  n(n+2\l) \left (1-\a_n^2 +\f 1{8n^2} \a_n^2 \right) < C^{-1}-\va,\quad \forall n \ge N_0.
$$
Here and in what follows, we write $\a_n$ for $\a_n^\l$ whenever it causes no confusion.
Since $\a_n\sim 1$ for $n$ sufficiently large and $\a_n \to 1$ when $n \to\infty$, we may
choose $N_0$ sufficiently large so that
\begin{equation}\label{3-15}
\left (1- \f1{C n(n+2\l)}\right)\f 1{\a_n^2} \leq \left( 1- \f {C^{-1}-\va}{n(n+2\l)}\right)
    \f 1{\a_n^2} - \f \va{8n^2} \leq 1-\f {1+\va} {8n^2}
\end{equation}
whenever $n\ge N_0$.

Let $\wh b_n$ be a sequence of nonnegative numbers such that $\sum_{n=N_0}^\infty \wh b_n^2  < \infty$.
We consider the function
$$
    f(t) = \sum_{n=N_0}^\infty \wh b_n [h_n^\l]^{-\f12} C_n^\l(t)(n(n+2\l))^{-\f12}, \quad -1 \le t\le 1.
$$
On the one hand, since $[h_n^\l]^{-\f12}C_n^\l(t)$ is orthonormal in $L^2(w_\l,[-1,1])$,
$$
     \|f\|_2^2 =\sum_{n=N_0}^\infty |\wh b_n|^2 (n(n+2\l))^{-1}.
$$
On the other hand, since $(-D_\l)^{1/2}f (t) = \sum_{n=N_0}^\infty \wh b_n [h_n^\l]^{-\f12}
C_n^\l(t)$, using \eqref{3-4} and the fact that $\ga_n^n=\a_n\a_{n+1}$, we obtain that
\begin{align*}
 c_\l \int_{-1}^1(1-t) |(-D_\l )^{\f12} f(t)|^2 dt =
    \sum_{n=N_0}^\infty |\wh{b}_n|^2 -\sum_{n=N_0}^\infty \a_n \a_{n+1} \wh b_n \wh b_{n+1}.
\end{align*}
Therefore, if \eqref{eq:HR-Gegen2} holds, we conclude that
\begin{align*}
\sum_{n=N_0}^\infty |\wh{b}_n|^2 -\sum_{n=N_0}^\infty \a_n \a_{n+1} \wh{b}_n \wh{b}_{n+1} \ge C^{-1} \sum_{n=N_0}^\infty \f 1{n(n+2\l)} |\wh{b}_n|^2,
\end{align*}
or equivalently, setting $\wh{g}_n =\a_n \wh{b}_n$, that
$$
\sum_{n=N_0}^\infty \left( 1- \f {1} { C n(n+2\l)} \right) \a_n^{-2} |\wh{g}_n|^2
     \ge \sum_{n=N_0}^\infty \wh{g}_n \wh{g}_{n+1}.
$$
By \eqref{3-15}, this implies that
$$
    \sum_{n=N_0}^\infty \wh{g}_n \wh{g}_{n+1} \leq \sum_{n=N_0}^\infty
            \left( 1- \f {1+\va}{8n^2} \right) |\wh{g}_n|^2,
$$
which becomes, upon rearranging terms,
\begin{equation}\label{3-16-1}
\sum_{n=N_0}^\infty \f {1+\va}{4n^2} |\wh{g}_n|^2 \leq
   |\wh{g}_{N_0}|^2 +\sum_{n=N_0}^\infty \left (\wh{g}_n-\wh{g}_{n+1}\right )^2.
\end{equation}
By the definition of $\wh g_n$ and the assumption on $\wh b_n$, using the fact that
$\a_n\sim 1$ for $n$ sufficiently large, the inequality \eqref{3-16-1} holds for an arbitrary
sequence of nonnegative numbers $\wh{g}_n$  satisfying $\sum_{n=N_0}^\infty |\wh{g}_n|^2 <\infty$.

Now for a given sufficiently large integer  $N\ge 2N_0$, we define
$$\wh{g}_n :=\begin{cases}
\sqrt{n},&\   \  \text{if $N_0\leq n\leq N$};\\
\sqrt{N-\f nN+1},&\   \  \text{if $N<n\leq N^2+N$};\\
0,&\   \  \text{if $n>N^2+N$ or $n<N_0$}.\end{cases}
$$
Then,  on the one hand,  a direct calculation shows that
\begin{align*}
\sum_{n=N_0}^\infty \f {1+\va}{4n^2} |\wh{g}_n|^2\ge \f {1+\va}4 \sum_{n=N_0}^ { N} \f 1n =\f {1+\va}4 \log N+O(1),\   \ \text{as $N\to\infty$},
\end{align*}
whereas on the other hand,
\begin{align*}
 \sum_{n=N_0}^\infty (\wh{g}_n-\wh{g}_{n+1})^2 & \leq \sum_{n=N_0}^{N-1} (\sqrt{n}-\sqrt{n+1})^2 +
   \sum_{k=0}^{N^2-1} \Biggl( \sqrt{N-\f kN}-\sqrt{N-\f {k+1}N}\Biggr)^2 \\
  &\leq \sum_{n=N_0}^{N-1} \Big(\f 1{2\sqrt{n}}\Big)^2+ \sum_{k=0}^{N^2-1} \Biggl( \f 1N \f 1{ \sqrt{N-\f kN}}\Biggr)^2 \\
  &  =\f 14 \log N +O(N^{-1}\log N)
\end{align*}
as $N\to\infty$. Therefore, by \eqref{3-16-1}, we conclude that
$$
   \f {1+\va}4 \log N \leq \f 14\log N +O(1),
$$
which, however, cannot hold for sufficiently large $N$.

We now prove sufficiency. Using the fact that $\g_n^n = \a_n \a_{n+1}$, we derive from
\eqref{eq:s-GegenSeries} and Lemma \ref{lem-3-1} that
\begin{align*}
  c_\l \int_{-1}^1 s \left | (-D_\l)^{\f12} f(s)\right|^2 w_\l(s) ds & =
      \sum_{n=n_0+1}^\infty \a_n \a_{n+1} \wh g_n^\l \wh g_{n+1}^\l \\
   & \le \sum_{n=n_0+1}^\infty \left(1-\frac{1}{8n^2}\right ) \a_n^2 | \wh g_n^\l |^2
\end{align*}
where $\wh g_n^\l = \wh f_n^\l \sqrt{n(n+2\l)}$. Consequently, for $J_\l(f)$ as in \eqref{J(f)-Gegen},
\begin{align*}
  J_n(f) \ge  \sum_{n=n_0+1}^\infty |\wh g_n^\l|^2 - \left(1-\frac{1}{8n^2}\right ) \a_n^2 | \wh g_n^\l |^2
     = \sum_{n=n_0+1}^\infty \beta_\l(n) | \wh f_n^\l |^2.
\end{align*}
Consequently, by Lemma \ref{lem:key-lemma},
$$
   J_n(f) \ge \b_\l(\infty)  \sum_{n=n_0+1}^\infty | \wh f_n^\l |^2 = \f18 (2\l-1)^2 \|f\|_2^2,
$$
which is the desired inequality \eqref{eq:HR-Gegen2} 
with $C = C_\l$.
\end{proof}

\begin{rem}
By Theorem \ref{thm:HR-Gegen-first}, Lemma \ref{lem:key-lemma} and Theorem \ref{thm:HR-iff},
the Hardy-Rellich inequality \eqref{eq:HR-Gegen2} holds for $n(\l) =0$ and optimal constant
if $0 < \l \le 1$ and $\l =3/2$. The numerical computation suggests that this should be true for
$1< \l < \l_0$, where $\l_0 \approx 1.8258$, which requires strengthening  (v) of Lemma
\ref{lem:key-lemma} to $n(\l) =0$ for $1 < \l \le \l_0$. %We do not know if the constant $8/(2\l-1)^2$
%holds in general.
\end{rem}

We are now in a position to discuss the optimal constant in the Hardy-Rellich inequality on
the sphere. For convenience, we set, for a given integer $k\in \NN$,
$$
  L_k^2(\sph):=\Bigl\{ f\in L^2(\sph): \int_{\sph} f(x)P(x)\, d\s(x)=0, \quad \forall P\in\Pi_k^d\Bigr\}.
$$

\begin{thm}\label{thm:HR-sharp-sphere}
The following assertion holds:
\begin{enumerate}[ \rm (i)]
\item For $d \ge 4$, there exists a positive integer $n(d)$, $n(d) \le 3 (d-2)^3 /16$, such that
for all $f \in L_{n(d)}^2(\sph) \cap  W_2^1(\sph)$,
\begin{equation} \label{eq:main-ineq}
   \int_{\sph} \left | f (x)\right|^2 d\s(x) \leq C_d  \min_{e\in\sph}
        \int_{\sph} (1- \la x, e\ra) |(-\Delta_0)^{\f12}f(x)|^2 d\s(x),
\end{equation}
where $C_d = \frac{8}{(d-3)^2}$ is optimal.
\item $n(2) = n(4) = n(5)=0$ and $n(6) = 4$.
\item For $d =3$, the inequality \eqref{eq:main-ineq} fails to hold for any finite constant $C_d$.
\end{enumerate}
\end{thm}

\begin{proof}
As in the proof of Theorem \ref{thm:main-HR}, we may assume that $e = (1,0,\ldots,0)$.
Since $f \in L_{n(d)}^2(\sph)$, $\wh f_{j,k}^n = 0$ for $n \le n(d)$.  Using Lemma \ref{lem-3-2}
and the fact that $\g_k^n \le \g_n^n$ for $0 \le k \le n$, we obtain
\begin{align*}
&\f1 {\o_d} \int_{\sph} x_1 \left |(-\Delta_0)^{\f12} f(x)\right |^2 d\s(x) \le \sum_{n=n(d)+1}^\infty
    \g_n^n  \sum_{k=0}^n \sum_{1 \le j \le a_{n-k}^{d-1}} |\wh g_{j,k}^n\wh g_{j,k+1}^{n+1}|,
\end{align*}
with  $\wh{g}_{j,k}^n=\sqrt{n(n+2\l)}\wh{f}_{j,k}^n$.  In analogy to \eqref{eq:key-d=4},
we use $\g_n^n=\a_n\a_{n+1}$, the Cauchy-Schwartz inequality and Lemma \ref{lem-3-1}
to conclude
\begin{align*}
  \f1 {\o_d} \int_{\sph} x_1  & \left |(-\Delta_0)^{\f12}  f(x)\right |^2 d\s(x) \\
 &\leq  \sum_{n=n(d)+1}^\infty \a_n\a_{n+1} \Biggl( \sum_{k=0}^{n} \sum_{j=1}^{a_{n-k}^{d-1}}
    \left|\wh{g}_{j,k}^n \right|^2\Biggr)^{\f12}
\Biggl( \sum_{k=0}^{n+1} \sum_{j=1}^{a_{n+1-k}^{d-1}} |\wh{g}_{j,k}^{n+1}|^2\Biggr)^{\f12} \\
 & \le \sum_{n= n(d)+1}^\infty \a_n^2 \left (1-\f 1{8n^2}\right )\sum_{k=0}^{n}
        \sum_{j=1}^{a_{n-k}^{d-1}} |\wh{g}_{j,k}^n|^2
\end{align*}
where  Lemma \ref{lem-3-1} is applied on $a_n=\a_n\Bigl( \sum_{k=0}^{n}
\sum_{j=1}^{a_{n-k}^{d-1}}|\wh{g}_{j,k}^n|^2\Bigr)^{\f12}$. Hence, for $J(f)$ defined
in \eqref{eq:J(f)}, we obtain
\begin{align*}
 J(f) \ge \sum_{n=n(d)}^\infty \left [1-\a_n^2 \left(1-\f 1{8n^2}\right)\right]
        \sum_{k=0}^{n} \sum_{j=1}^{a_{n-k}^{d-1}} |\wh{g}_{j,k}^n|^2.
\end{align*}
We choose $n(d)$ to be the integer $n(\l)$ with $\l = (d-2)/2$ in Lemma \ref{lem:key-lemma}.
By the definition of $\beta_\l(n)$, we conclude then
$$
  J(f) \ge \beta_\l(\infty) \sum_{n=1}^\infty \sum_{k=0}^{n} \sum_{1\leq j\leq a_{n-k}^{d-1}}|\wh{f}_{j,k}^n|^2
   = \frac18 (d-3)^2 \| f\|_2^2,
$$
which proves \eqref{eq:main-ineq}. Applying to functions of the form $f(x_1)$ for
$x = (x_1,\ldots, x_d) \in \sph$, the inequality \eqref{eq:main-ineq} becomes the inequality
\eqref{eq:HR-Gegen2} for the Gegenbauer weight function with $\l  = (d-2)/2$, from which
the optimality of the constant follows from Theorem \ref{thm:HR-iff}. This completes the proof
of (i). While (ii) follows immediately from Lemma \ref{lem:key-lemma}, the same argument
for the optimal constant in (i) also proves (iii) by Theorem \ref{thm:HR-iff}.
\end{proof}

The proof of the above theorem can also be used to determine a constant in the Hardy-Rellich
inequality. Indeed, it yields the following corollary:

\begin{cor}
Let $d \ge 4$. If $\tau_d:= \min_{n \ge 1} \tau_\l(n) > 0$, where $\l = (d-2)/2$,  then the
Hardy-Rellich inequality \eqref{eq:HR-Gegen2} holds for all $f\in L_0^2(\sph)\cap W_2^1(\sph)$ with $C =\tau_d^{-1}$. In particular,
$\tau_6 = \frac{141}{128}$ and
$$
   \tau_d = \beta_\l(1) = (d-1) \left(1 - \frac{7 \sqrt{\pi} \Gamma(\f d 2)}{4d \Gamma(\f{d-1}{2}})\right),
$$
for $d = 7,8, 9,10$.
\end{cor}

In fact, we only need to verify that $\tau_d$ has the stated value. By Lemma \ref{lem:key-lemma},
we only need to compare the values of $\beta_n(\l)$ for $ n \le 3 \l^3/2$ with that of
$\beta_\l(\infty)$, which can be verified numerically for small $d$. The result shows that
$$
   \tau_6 =  \beta_2 (2) = \frac{141}{128} < \frac{9}{8} = \beta_2(\infty),
$$
and for $d \ge 7$, $\tau_d = \beta_\l(1)$.

We expect that the corollary holds for all $d \ge 10$. However, a more interesting question
is that if
$$
        C_d =  \frac{8}{(d-3)^2} = (\beta_\l(\infty))^{-1} <  \tau_d^{-1}, \quad d \ge 6,
$$
is the optimal constant for the Hardy-Rellich inequality with $f\in L_0^2(\sph)\cap W_2^1(\sph)$. We have proved
that it is for $d = 2, 4, 5$. Thus, the question of finding the optimal constant  remains open  for $d \ge 6$.

\section{Uncertainty principles}
\setcounter{equation}{0}

Our uncertainty principle follows as an application of the Hardy-Rellich inequality in
the previous section.

\begin{thm}\label{thm-2-1}
Let $f\in W_2^1(\sph)$ be such that $\int_{\sph} f(y) d\s(y)=0$ and $\|f\|_2=1$.  If $d \ge 2$ then
\begin{equation}\label{UC-sphere}
\min_{e\in\sph}  \left[\f 1{\o_d} \int_{\sph} (1-\la x,  e\ra) |f(x)|^2 \, d\s(x) \right] \|\nabla_0 f\|_2^2
   \ge B_d
\end{equation}
where the constant $B_d$ is given by
\begin{equation}\label{eq:Bd}
   B_d =  (d-1) \left(1 - \frac{2}{\sqrt{d+3}}\right), \qquad d \ge 3,
\end{equation}
and, alternatively, for $d \ne 3$, $B_d = C_d^{-1}$ with $C_d$ being the constant in the
Hardy-Rellich inequality. In particular, $B_2 = 1/8$ and $1/8$ is sharp.
\end{thm}

\begin{proof}
Since $\int_{\sph} f(y)\, d\s(y)=0$,  $(-\Delta_0)^{\f 12}(-\Delta_0)^{-\f12} f =f$. Thus, using the Cauchy-Schwartz inequality, we have that
\begin{align*}
1=\|f\|_2^2 &=\f 1{\o_d} \int_{\sph} \bigl[(-\Delta_0)^{\f 12}f (x) \bigr]\bigl[  (-\Delta_0)^{-\f12} f(x)\bigr]\, d\s(x) \\
&\leq \|(-\Delta_0)^{-\f12} f\|_2\|(-\Delta_0)^{\f12} f\|_2,\end{align*}
which, by \eqref{eq:main-ineq} applied to $(-\Delta_0)^{\f12} f$ instead of $f$, is estimated by
$$
C_d\|(-\Delta_0)^{\f12} f\|_2\min_{e\in \sph} \int_{\sph} (1- \la x, e\ra) |f(x)|^2\, d\s(x),\   \   \ d\neq 3.
$$
This together with \eqref{1-6} implies the desired inequality for  $d\neq 3$. For the sharpness of
the constant $B_2 = 1/8$, see \eqref{eq:Bd-bounds} below.

Next we give a different proof of \eqref{UC-sphere} that  covers the case of $d = 3$ as well. Define
the differential operators
$$
       D_{i,j}=x_i\p_j-x_j\p_i,  \qquad 1\leq i\neq j\leq d.
$$
We shall use  the following two identities about these differential operators:
\begin{enumerate}[\rm (i)]
  \item For $f, g\in C^1(\sph)$, and $1\leq i\neq j\leq d$,
 \begin{equation}\label{integrbyparts}
    \int_{\sph} D_{i,j} f(x) g(x)\, d\s(x)=-\int_{\sph}f(x) D_{i,j} g(x)\, d\s(x).
 \end{equation}
 \item For $f\in C^1(\sph)$,
 \begin{equation}\label{gradient}
   |\nabla_0 f(x)|^2 =\sum_{1\leq i<j\leq d}|D_{i,j}f(x)|^2,\   \   x\in\sph.
 \end{equation}
 \end{enumerate}
These two identities can be found in \cite[Chapter 1]{DaiX}, and they can be also easily
verified by straightforward calculations.

%Given  $x\in\mathbb{S}^2$ and  $\t\in (0,\pi)$, we
%denote by  $B(x,\t)$  the spherical cap $\{y\in \mathbb{S}^2:\  \  \arccos \la x, y\ra \leq \t\}$.
Without loss of generality, we may assume that the minimum is achieved at  $e=(1,0,\ldots, 0)$.
For convenience, we set
$$
r:=\f 1{\o_d}\int_{\sph} (1-x_1)|f(x)|^2\, d\s(x) \quad \hbox{and} \quad
     Lf: =  r \|\nabla_0 f\|_2^2.
$$
Our goal is to show that $ L f \ge B_d$. Since $\|f\|_2 =1$, it is evident that $r\in (0,2)$.
Using \eqref{integrbyparts} and the fact that $D_{1,j}x_j=x_1$ for $j \ge 2$, it follows readily that
\begin{align}\label{4-4}
   \f 1{\o_d} \int_{\sph} &\, \bigg(\sum_{ j=2}^d x_j D_{1,j}f(x) \bigg) f(x)\, d\s(x)
   \\
  & = -\f {d-1}2 \f{1}{\o_d}\int_{\sph} x_1 |f(x)|^2\, d\s(x) \notag
   =  -\f {d-1}2 (1-r).
\end{align}
Using \eqref{gradient} and  the fact that $\|x\| =1$, we see that
\begin{align*}
\bigg| \sum_{ j=2}^d x_j D_{1,j}f(x) \bigg |^2 &\leq \Bl(\sum_{j=2}^d x_j^2 \Br)\Bl(\sum_{j=2}^d |D_{1,j}f(x)|^2\Br)
       \leq (1-x_1^2) \|\nabla_0 f(x)\|^2,
\end{align*}
which implies, by \eqref{4-4} and the Cauchy-Schwartz inequality,
\begin{align} \label{4-5-2}
  \frac{(d-1)^2}{4}|1-r |^2  & \leq \Bl(\f1{\o_d}\int_{\sph} |\sum_{j=2}^d x_j D_{1,j} f(x)|^2\f1 {1-x_1^2}\, d\s(x)\Br)\\
      & \qquad \times \Bl(\f1{\o_d}\int_{\sph} |f(x)|^2 (1-x_1^2)\, d\s(x)\Br) \notag  \\
   &\leq \|\nabla_0 f\|_2^2 \f1{\o_d}\int_{\sph} |f(x)|^2 (1-x_1^2)\, d\s(x).
\notag
\end{align}
Using again $\|f\|_2^2 =1$, the Cauchy-Schwartz inequality shows that
\begin{align} \label{4-5-2-2}
\f1{\o_d}\int_{\sph} |f(x)|^2 x_1^2d\s(x) \ge
    \Bl |\f1{\o_d}\int_{\sph} |f(x)|^2 x_1\, d\s(x)\Br|^2 = (1-r)^2,
\end{align}
from which it follows that
\begin{align*}
&    \f1{\o_d}\int_{\sph} |f(x)|^2 (1-x_1^2)\, d\s(x)\leq 1- (1-r)^2 = (2-r)r.
\end{align*}
Thus, by \eqref{4-5-2}, we conclude that
\begin{align*}
   \frac{(d-1)^2}{4}(1-r)^2 \le (2-r) r \|\nabla_0 f\|_2^2 = (2-r) Lf,
\end{align*}
or equivalently,
\begin{align}
   Lf \ge   \frac{(d-1)^2}{4} \frac{(1-r)^2}{2-r}.  \label{4-5-0}
\end{align}
On the other hand, by \eqref{s-expansions}, \eqref{1:ch0} and the assumption that
$\int_{\sph} f(x) d\s(x)=0$,
\begin{align*}
 1 =  \|f\|_2^2 =   \sum_{n=1}^\infty \|\proj_n f\|_2^2
  \le  \frac1 {d-1} \sum_{n=1}^\infty n(n + d-2 ) \|\proj_n f\|_2^2  = \f{1}{d-1}\|\nabla_0 f\|_2^2.
\end{align*}
Hence, it follows that $Lf  = r  \|\nabla_0 f\|_2^2 \ge (d-1)r$. Together with \eqref{4-5-0},
we have shown that
\begin{align*}
    L f   \ge (d-1) \max \left\{ \frac{d-1}{4} \frac {(1-r)^2}{2-r}, r \right\}
       \ge (d-1)\min_{t\in (0,2)} \max \left\{ \frac{d-1}{4} \frac {(1-t)^2}{2-t}, t \right\}.
%     =2-\f{2\sqrt{6}}3.
\end{align*}
Finally, choosing $t\in (0,2)$ such that  
$\frac{d-1}{4} \frac {(1-t)^2}{2-t}= t$, we obtain  \eqref{eq:Bd}.
\end{proof}

\begin{rem} \label{remark4.1}
The constant $B_d$ obtained via the Hardy-Rellich inequality is $(d-3)^2/8$ for $d = 2,4,5$
and for the restricted class of $L_{n(d)}^2(\sph) \cap W_2^1(\sph)$. For $d = 4,5$ this is
worse than the constant $B_d$ in \eqref{eq:Bd}. On the other hand, when $d \to \infty$,
$B_d = d -1 + \CO(\sqrt{d})$ in \eqref{eq:Bd}, which can be improved to
$B_d = n(d) d + \CO(\sqrt{d})$ in the restricted class of $L_{n(d)}^2(\sph) \cap W_2^1(\sph)$,
and it is worse in the order of magnitude for large $d$.
\end{rem}

The same idea of this proof also yields the following inequality in $L^2(w_\l,[-1,1])$.

\begin{cor}
Let $\l > -1/2$. For $f\in W_2^1([-1,1])$ such that $\int_{-1}^1 f(y) w_\l(y) dy = 0$
and $\|f\|_{\l,2} =1$, there is a positive constant $B_\l$ such that
\begin{equation}\label{eq:uc-gegebauer}
   \int_{-1}^1 (1-t) |f(t)|^2 w_\l(t) dt  \int_{-1}^1 \left| (-D_\l)^{\f12} f(t)\right |^2 w_\l(t) dt \ge B_\l,
\end{equation}
where $B_\l=2-\f{2\sqrt{6}}3$ for $\l=\f12$, and  $B_\l=C_\l^{-1}$ for $\l\neq \f12$ with $C_\l$
being the constant in the Hardy-Rellich inequality. In particular, for $0 \le \l \le 3/2$ and $\l \ne 1/2$,
$B_\l = (2 \l -1)^2/8$.
\end{cor}

The quantity on the left hand side of \eqref{UC-sphere} is related to the following vector in $\RR^d$:
\begin{equation*}
    \tau (f):=\int_{\sph} x |f(x)|^2 \, d\s(x).
\end{equation*}
       The norm of the vector $\tau(f)$ in $\RR^d$ is denoted by $\|\tau(f)\|$. We  observe that
\begin{equation}\label{3-12-0}
          \|\tau (f)\|\leq \int_{\sph} |f(x)|^2\, d\s(x) = \|f\|_2^2.
\end{equation}

\begin{cor} \label{prop4.2}
Let $f\in W_2^1(\sph)$ be such that $\int_{\sph} f(y)\, d\s(y)=0$ and $\|f\|_2=1$.  If $d \ge 2$,  then
\begin{equation}\label{3-11-0}
     (1-\|\tau(f)\|) \|\nabla_0 f\|_{2}^2\ge C_d^{-1}.
\end{equation}
\end{cor}

\begin{proof}
Since $\|z\| = \max_{e \in \RR^d} \la z, e \ra$ for all $z \in \RR^d$, $\|\tau (f)\| =
\max_{e\in \sph} \la \tau(f), e \ra$, which shows that
\begin{equation}\label{tau-1}
    \| \tau (f)\| = \max_{e\in\sph}  \left[\f 1{\o_d} \int_{\sph} \la x,  e \ra |f(x)|^2 \, d\s(x) \right].
\end{equation}

Since $\|f\|_2 =1$, it follows that
 \begin{align} \label{1-||tau||}
   1-\|\tau(f)\| =\min_{e\in\sph}  \left [\f 1{\o_d}\int_{\sph} (1-\la x, e\ra) |f(x)|^2 d\s(x)\right].
\end{align}
 Thus, \eqref{3-11-0} is an equivalent form of \eqref{UC-sphere}.
\end{proof}

As in the case of the Hardy-Rellich inequality, the condition $\int_{\sph} f(x)d\s =0$ is necessary
 for the uncertainty principle inequalities stated above, as can be seen by setting $f(x) =1$.
This  restriction, however, can be removed to give the following new version of  uncertainty
principle.

\begin{thm}\label{cor-3-3}
Assume that $d\ge 2$ and let  $f\in W_2^1(\sph)$ be such that $\|f\|_{2}=1$. Then
\begin{equation}\label{3-11}
      (1-\|\tau(f)\|) \|\nabla_0 f\|_{2}^2 \ge c_d \|\tau(f)\|.
\end{equation}
\end{thm}

\begin{proof}
We first prove \eqref{3-11} for the case of $d\ge 4$.
Let $m_f$ denote the mean value of $f$, that is, $m_f:= \frac{1}{\o_d} \int_{\sph} f(x) d\s$.
Then $m_f \le \|f\|_2 \le 1$. By definition, $m_f = \proj_0 f$. By Cauchy-Schwartz inequality,
\begin{align}
 m_f^2 & \leq  \f 1{\o_d} \int_{\sph} |f(x)|^2 (1-\la x, e\ra)\, d\s(x)
 \f 1{\o_d} \int_{\sph} (1-\la x, e\ra)^{-1}\, d\s(x) \notag\\
 &=\f{d-2}{d-3}  \int_{\sph} |f(x)|^2 (1-\la x, e\ra)\, d\s(x),\label{3-13-eq}
\end{align}
since, for $d \ge 4$,
$$
\f 1{\o_d} \int_{\sph} \frac{ d\s(x)}{1-\la x, e\ra}= \f 1 {\o_d} \int_{-1}^1 \frac{(1-t^2)^{\f{d-3}2}}{1-t} dt
= \f {2^{d-3}\Ga(\f d2)\Ga ( \f {d-3}2)}{\Ga(d-2)}=\f{d-2}{d-3}.
$$
Now define $I f: = m_f +(-\Delta_0)^{-\f12}f$.  Since $\int_{\sph} (-\Delta_0)^{\pm \f12}f d\s = 0$
by definition, we have
$$
  \f 1{\o_d} \int_{\sph} If (x) \left( m_f +(-\Delta_0)^{\f12}f \right) d\s(x)  \ge \|f\|_2^2 =1.
$$
Applying the Hardy-Rellich inequality on $(-\Delta_0)^{-\f12}f$ and using \eqref{3-13-eq}, we deduce that
\begin{equation*}
\f 1{\o_d} \int_{\sph} |I f(x)|^2 \, d\s(x) \leq c \min_{e\in\sph} \int_{\sph} (1-\la x, e\ra) |f(x)|^2\, d\s(x),
\end{equation*}
where $c$ is an constant depending only on $d$. Consequently, it follows from the
Cauchy-Schwartz inequality that
\begin{align}
 1& \leq \|f\|_2^4 \leq \|I f\|_2^2 \left \| m_f + (-\Delta_0)^{\f12} f \right \|_2^2 \label{3-14-eq}\\
   &\leq c \min_{e \in \sph} \left (\f 1{\o_d} \int_{\sph}  (1-\la x, e\ra) |f(x)|^2 d\s(x)\right)\left(\|\nabla_0 f\|_2^2 + m_f^2\right).\notag
 \end{align}
Thus, if $|m_f|\leq 4 \|\nabla_0 f\|_{2}$, then desired inequality \eqref{3-11} follows directly from
\eqref{3-12-0} and \eqref{3-14-eq}. Thus, it remains to prove \eqref{3-11} under the additional
assumption that  $|m_f|> 4 \|\nabla_0 f\|_{2}$. To this end, we write $f= m_f + g$. Since
$m_f = \proj_0 f$,
\begin{align*}
   \|g\|_2 = \Bl(\sum_{n=1}^\infty \|\proj_n f\|_2^2\Br)^{1/2} \le \Bl(\sum_{n=1}^\infty n(n+\l) \|\proj_n f\|_2^2\Br)^{1/2} =
         \|\nabla_0 f\|_{2}\leq \f 14 |m_f|,
\end{align*}
which implies that $|m_f |= \|f - g\|_2 \ge \|f\|_2 - \|g\|_2 \ge 1 -  \f 14 |m_f|$, so that
$1\geq |m_f|\ge \f45$. Since $|f|^2=|m_f|^2 +2 m_f g +|g|^2$, it follows from \eqref{1-||tau||} that
\begin{align*}
& 1-\|\tau(f)\|   = \min_{e\in\sph} \int_{\sph} (1-\la x, e\ra) |f(x)|^2 \, d\s \notag\\
     & = m_f^2 +\min_{e\in\sph}\left[- 2 m_f\int_{\sph}\la x, e\ra g(x) d\s
             +\int_{\sph}(1-\la x, e\ra) |g(x)|^2d\s \right]
\end{align*}
since $\int_{\sph} g(x)\, d\s(x)=0$, from which it follows that
$$
1-\|\tau(f)\|  \ge  m_f^2 -2|m_f|\|g\|_{2}\ge \f 12 m_f^2\ge \f 8{25}.%\label{3-16}
$$
A similar argument also yields
\begin{align*}
\|\tau (f)\|&=\max_{e\in\sph} \int_{\sph} \la x, e\ra \left(m_f^2+ g^2+2 m_f g \right)\, d\s \notag\\
&= \max_{e\in\sph} \left( \int_{\sph} \la x, e\ra g^2(x) d\s(x) +2 m_f \int_{\sph} \la x, e\ra g(x) d\s(x)\right)\\
& \leq (2 |m_f|+1) \|g\|_{2}^2 \leq 3 \|\nabla_0 f\|_{2}^2. %\label{3-17}
\end{align*}
Thus, combining these two inequalities, we conclude that
\begin{align*}
  &(1-\|\tau(f)\|) \|\nabla_0 f\|_{2}^2\geq \f 8{25} \|\nabla_0 f\|_{2}^2\ge \f {8}{25}\cdot \f 13 \|\tau (f)\|.
\end{align*}
This proves \eqref{3-11} for $d\ge 4$.
Note that the only place in the above proof  where  the condition $d\ge 4$ is needed is  the
inequality \eqref{3-13-eq}.

Thus, it remains to prove that \eqref{3-11} holds for $d=2,3$. We shall consider the case of $d=3$ only, as the same proof below works equally well for the case $d=2$.
If \begin{equation*}%\label{4-15-0}
    m_f^2 \leq 25 (1-\|\tau(f)\|),
\end{equation*}
then by the
remark at  the end of the last paragraph,  the  proof for $d\ge 4$ with slight modifications works equally well for the case  $d=3$.   Thus, it suffices to prove the assertion for $d=3$  under the
additional assumption that
\begin{equation}\label{4-15-0}
    m_f^2 \ge 25 (1-\|\tau(f)\|).
\end{equation}
 Without loss of generality, we may assume that the supremum in \eqref{tau-1} is achieved at the point $e=(1,0,0)\in\mathbb{S}^2$ so that
$1-\|\tau (f)\|=\f1{4 \pi}\int_{\mathbb{S}^2} |f(x)|^2 (1-x_1)\, d\s(x)$.
 Thus, \eqref{4-15-0} implies that
 \begin{align*}
   1-\|\tau (f)\|&=
   \f1{4\pi}\int_{\mathbb{S}^2} (1-x_1)|f(x)|^2\, d\s(x)\\
   &\, \leq \f 1{25} m_f^2 \leq \f 1{25} \|f\|_2^2 \leq \f1{25}.
\end{align*}
By \eqref{4-5-0}  in the proof of Theorem \ref{thm-2-1} with $r=1-\|\tau (f)\|\leq \f1{25}$,  which does not require the
condition that $\int_{\mathbb{S}^2} f(x)\, d\s(x)=0$, we deduce that
\begin{align*}
    (1-\|\tau(f)\|)\|\nabla_0 f\|_2^2 = L f \ge \min_{t\in (0, \f 1{25})}\f {(1-t)^2 }{2-t}\ge c \ge c \|\tau(f)\|.\end{align*}
This completes the proof.
\end{proof}

Since, by \eqref{3-12-0}, $1-\|\tau(f)\|^2 \ge 1-\|\tau (f)\|$ and $\|\tau(f)|^2\leq \|\tau (f)\|$, it follows
as a corollary of Theorem \ref{cor-3-3} that
\begin{equation} \label{eq:UC-old}
\left (1-\|\tau(f)\|^2 \right) \|\nabla_0 f\|_2^2 \ge c_d\|\tau (f)\|^2.
\end{equation}
This inequality was called the uncertainty principle on the sphere and was discussed in several
papers in the literature \cite{NW, RV, Selig}. The inequality \eqref{eq:UC-old} is weaker than
\eqref{3-11} since it can be deduced from the latter. In fact, a simple proof of this inequality follows
from our proof of Theorem \ref{thm-2-1}.

\begin{cor}
If  $f \in  W_2^1(\sph)$, and $\|f\|_2=1$, then
\begin{equation} \label{eq:UC-old-2}
\left (1-\|\tau(f)\|^2 \right) \|\nabla_0 f\|_2^2 \ge  \left(\frac{d-1}{2}\right)^2 \|\tau (f)\|^2.
\end{equation}
\end{cor}

\begin{proof}Using \eqref{tau-1},
we can assume that $\|\tau(f)\| = \f{1}{\o_d} \int_{\sph} x_1 |f(x)|^2 d \s(x)$  without loss of generality.
With $r=1-\|\tau(f)\|$,  we can rewrite \eqref{4-5-0} as %equivalently as
$$ (2-r) r \|\nabla_0 f\|_2^2 \ge \f {(d-1)^2}4 (1-r)^2,$$
which is the desired inequality \eqref{eq:UC-old-2}.
\end{proof}

The constant $(\frac{d-1}{2})^2$ in \eqref{eq:UC-old-2} was shown to be optimal in \cite{RV} by using
the heat kernel defined by
\begin{equation}\label{heat-kernel}
   q_{t}^\l(s) : =  \sum_{n=1}^\infty e^{-n(n+2\l) t} \frac{n+\l}{\l} C_n^\l(s).
\end{equation}
Indeed,
the computation in \cite{RV} shows that $\tau(q_t^\l(s))/\|q_t^\l(s)\|_2 \to 1$ as $t \to 0+$, where
$\|\cdot\|_2$ denotes the $L^2(w_\l;[-1,1])$ norm, and
$$
\lim_{t \to 0+} \frac{  \| \sqrt{1-\{\cdot\}} q_t^\l \|_2^2} {\|q_t^\l \|_2^2}  = \frac12 \left(\l + \f12\right)
   \quad \hbox{and}\quad
\lim_{t \to 0+} \frac{ \|(-D_\l)^{\f12} q_t^\l \|_2^2} {\|q_t^\l \|_2^2}  = \l + \f12.
$$
Setting $f(x) = q_t^\l( \la x, e \ra)$ then shows the optimality of the constant in \eqref{eq:UC-old-2}.

We end up this section with the following remark. Our proof of Theorem \ref{thm-2-1} does not lead
to the optimal constants in these inequalities, since the proof based on the Hardy-Rellich inequality
as well as  the H\"older inequality with $F = (-\Delta_0)^{\f12}f$ and $G= (-\Delta_0)^{-\f12}f$, whereas
the constant in the second proof is discussed in Remark \ref{remark4.1}. If we set $f =
q_t^\l / \|q_t\|_2$ in \eqref{eq:uc-gegebauer} and letting $t \to 0+$, then we obtain $B_\l
\le (2\l+1)^2/8$. In particular, for the optimal  constant in $B_d$ in \eqref{UC-sphere}, we
conclude, together with Theorem \ref{thm-2-1}, that
\begin{equation}\label{eq:Bd-bounds}
               \frac{(d-3)^2}{8} \le  B_d \le \frac{(d-1)^2}{8}
\end{equation}
for $d =2, 4, 5$. In particular, this shows that the constant $B_2 =1/8$ is optimal for the
inequality \eqref{UC-sphere} for $d =2$. Furthermore, setting $f(x) = q_t^\l (\la x,e\ra)$
and letting $t \to 0+$ in \eqref{3-11} shows that that the constant in \eqref{3-11} satisfies
$c_d \le (d-1)^2 /8$.

\section{Appendix: Proof of Lemma \ref{lem:key-lemma}}
\setcounter{equation}{0}

The item (i) of the lemma follows from a straightforward calculation.
For (ii), we let
$$
   \Phi_\l(x) = \frac{ \Gamma(x + 1)\Gamma( x + 1/2 + \l)}{(x + \l/2) \Gamma(x + 1/2)\Gamma(x + \l)}.
$$
Then it is easy to verify that $\Phi_\l(n) = \a_{2n}^2$ and $\Phi_\l(n+1/2) = \a_{2n+1}^2$. A direct
computation shows that
$$
  \frac{\Phi_\l(x+1)}{\Phi_\l(x)} = 1 + \frac{\l (\l-1)}{(x+\l)(2x+1)(2 x + \l +2)},
$$
from which the monotonicity of $\a_{2n}$ and $\a_{2n+2}$ follows readily.

For the proof of (iii), we define
\begin{equation} \label{eq:Psi-defn}
  \Psi_\l(x) =  \Bl(1 - \Phi_\l(x) + \frac{1}{32 x^2} \Phi_\l(x) \Br) x (x + \l).
\end{equation}
It is easy to verify then that
$$
\beta_\l(2n) = 4 \Psi_\l(n) \quad \hbox{and} \quad \beta_\l(2n+1) = 4 \Psi_\l(n+1/2).
$$
Using the following formula with $c=\f12$ and $z=n+\f{1+\l}2$,
\begin{align*}
z^{-c} \f { \Ga(z+a+c)}{\Ga(z+a)}= 1 +\f { c ( 2a+c-1)}{2z}
 + \f { c ( c-1) \bigl[ 3(2a+c-1)^2-c-1\bigr]}{24z^2} +O(z^{-3})
\end{align*}
as $z \to \infty$, a straightforward calculation shows that
$$
   \Phi_\l(x) = 1 + \frac{ \l- \l^2 + \f14} {8 x^2}+ O(x^{-3}).
$$
Substituting this asymptotic formula into \eqref{eq:Psi-defn}, the limit in (iii) follows readily.

To prove (iv), we rewrite, after a direct computation, that
$$
\Psi_\l(x) =  x (x + \l) - \frac{ (x+\l)(32 x^2 -1)}{16 (2x+\l)}G_\l(x),
$$
where the function $G_\l$ is given by
$$
   G_\l(x) =  \frac{ \Gamma(x)\Gamma( x + 1/2 + \l)}{\Gamma(x + 1/2)\Gamma(x + \l)}
         = {}_2F_1\left( \begin{matrix} -\f12, -\l \\ x \end{matrix} ;1 \right)
$$
in terms of the hypergeometric function ${}_2F_1$.  Then (iv) is a consequence of the
following proposition.

\begin{prop} \label{prop:appendix}
For $x \ge 3 \l^3$, $\Psi_\l(x+1) <\Psi_\l(x) $. In particular, $\{\beta_\l(2n)\}$ and
$\{\beta_\l(2n+1)\}$ are both decreasing for $n \ge 3 \l^3/2$.
\end{prop}

\begin{proof}
We consider the difference operator $\Delta f(x) = f(x+1) - f(x)$ and $\Delta^{r+1} = \Delta^r \Delta$
for $r = 2,3,...$. From the definition, it shows
\begin{equation} \label{eq:D-Psi}
      \Delta  \Psi_\l(x) = 2x+\l +1 + A (x) G_\l(x),
\end{equation}
where
\begin{align}\label{eq:A(x)}
A(x) =  \frac{(x + \l) (32 x^2 -1)} {16 (2 x+\l)}  -\frac{ x (x+\l+1) (2x+2\l+1) (32 (x+1)^2 -1)}{
     16 (x + \l) (2 x+1) (2 x+\l+2)}.
\end{align}
Taking two more differences gives, with the help of a computer algebra system (we used the
{\it Mathematica}), that
$$
   \Delta^3 \Psi_\l (x) = \frac{F_\l(x)} {128 (\l + 2 x) (2 + \l + 2 x) (4 + \l + 2 x) (6 + \l + 2 x)} \frac{
     \Gamma(x) \Gamma(x+\l + \f12)}{\Gamma(x+\f72) \Gamma(x+\l+3)},
$$
where
\begin{align*}
F_\l(x) = & -\l (1 + \l) (2 + \l) (4 + \l) (37 - 77 \l + 37 \l^2) \\
&  +( -568 + 308 \l  + 1346 \l^2 + 325 \l^3 - 574 \l^4 - 501 \l^5)x \\
&+ 4 (72 + 386 \l + 3 \l^2 - 280 \l^3 - 227 \l^4 + 48 \l^5 + 8 \l^6)x^2 \\
& + 4 (270 - 97 \l - 394 \l^2 - 251 \l^3 + 80 \l^4 + 104 \l^5)x^3 \\
& + 16 (-10 - 79 \l - 49 \l^2 - 8 \l^3 + 48 \l^4)x^4 \\
& + 128 (-4 - 3 \l - 2 \l^2 + 3 \l^3)x^5 -128 x^6.
\end{align*}
We show that if $x \ge 3 \l^3$, then  $F_\l(x) \le 0$ so that $\Delta^3 \Psi_\l(x) \le 0$. This relies
on the following expression of $F_\l$, computed by the {\it Mathematica},
\begin{align*}
F_\l(x) =&  -128 (x - 3 \l^3) x^5 - 128 (4 + 3 \l + 2 \l^2) (x - 3 \l^3) x^4 \\
&  -  16 (10 + 79 \l + 49 \l^2 + 104 \l^3 + 24 \l^4 + 48 \l^5) (x -  3 \l^3) x^3 \\
& -  4 (-270 + 97 \l + 394 \l^2 + 371 \l^3 + 868 \l^4 + 484 \l^5 + 1248 \l^6 +
    288 \l^7   \\
  & \qquad\qquad\qquad  + 576 \l^8) (x - 3 \l^3) x^2 \\
 & -  4 (-72 - 386 \l - 3 \l^2 - 530 \l^3 + 518 \l^4 + 1134 \l^5 + 1105 \l^6 +
    2604 \l^7 \\
 & \qquad   + 1452 \l^8 + 3744 \l^9   + 864 \l^{10} + 1728 {\l^{11}}) (x -  3 \l^3) x \\
& -(568 - 308 \l - 1346 \l^2 - 1189 \l^3 - 4058 \l^4 + 465 \l^5 - 6360 \l^6 +
    6216 \l^7 \\
   & \qquad  + 13608 \l^8 + 13260 \l^9 + 31248 \l^{10} + 17424 \l^{11} +
    44928 \l^{12} + 10368 \l^{13}\\
   & \qquad  + 20736 \l^{14}) x  -\l (1 + \l) (2 + \l) (4 + \l) (37 - 77 \l + 37 \l^2).
 \end{align*}
If $x \ge 3 \l^3$, then every term in the right hand side of the above expression is
negative, so that $F_\l(x) $, hence $\Delta^3 \Psi_\l(x)$, is negative if $x \ge 3 \l^3$.
By the definition of $\Delta$, it follows that $\Delta^2 \Psi_\l(x ) \ge \Delta^2\Psi_\l(x+1)$
for  $x \ge 3 \l^3$. Since the limit of $\Psi_\l(x)$ as $x \to \infty$ is finite, $\Delta^r \Psi_\l(x)
\to 0$ as $x \to \infty$. In particular, $\lim_{x \to \infty} \Delta^2 \Psi_\l(x) =0$, so that
$\Delta^2 \Psi(x) \ge 0$ for $x \ge 3 \l^3$. The same argument implies then
$\Delta \Psi_\l(x) \le \Delta \Psi_\l(x+1) \le 0$, which shows, in turn, that $\Psi_\l(x+1) \le \Psi_\l(x)$
for $x \ge 3 \l^3$ as desired.
\end{proof}

We further conjecture that the condition $n \ge 3 \l^3/2$ in the above proposition is not needed
for $1/2 \le \l \le 3/2$. For $\l = 1/2,1,3/2,2$, this can be verified by evaluating $b_\l(n)$
numerically, which proves (v) of Lemma \ref{lem:key-lemma}.

Let us note that a more careful computation of the Proposition \ref{prop:appendix} shows
that we could improve the condition $x \ge 3 \l^3$ somewhat, say to $x \ge 3\l^3 - c \l^2$ for
 some $c > 0$. However, the region on which $\Delta^3 \Psi_\l(x) < 0$ is a subset of the region
on which $\Psi_\l(x)$ is monotonically decreasing. Determining the cut-off point $x_0$ so that
$\Psi_\l(x)$ is decreasing for $x \ge x_0$ appears to be not so easy.

\bibliographystyle{amsalpha}

 \newpage
 
\begin{center}
{\bf { \sc \large Erratum: The  Hardy-Rellich inequality and uncertainty principle on the sphere }}
\end{center}

\vskip 0.3in

\begin{abstract}
The text below is the erratum submitted to Constructive 
Approximation.
\end{abstract}
 
\bigskip

Several forms of uncertainty principles on the unit sphere are established in \cite{DaiX14}. When stated in term of the vector
\begin{equation*}
    \tau (f):=\f 1{\o_d}\int_{\sph} x |f(x)|^2 \, d\s(x)
\end{equation*}
of $\RR^d$ (normalization constant $1/\o_d$ was missing in \cite{DaiX14}), our main result is in

\medskip\noindent
{\bf Corollary 4.4}
{\it Let $f\in W_2^1(\sph)$ be such that $\int_{\sph} f(y)\, d\s(y)=0$ and $\|f\|_2=1$.  If $d \ge 2$,  then}
\begin{equation}%\label{3-11-0}
     (1-\|\tau(f)\|) \|\nabla_0 f\|_{2}^2\ge C_d^{-1}. \tag{4.11}
\end{equation}
\medskip

Here $C_d$ is a constant given in Theorem 4.1. We next attempted to remove the condition that
$\int_{\sph} f(y)\, d\s(y)=0$ and stated

\medskip \noindent
{\bf Theorem 4.5}  {\it Assume that $d\ge 2$ and let  $f\in W_2^1(\sph)$ be such that $\|f\|_{2}=1$. Then}
\begin{equation}
      (1-\|\tau(f)\|) \|\nabla_0 f\|_{2}^2 \ge c_d \|\tau(f)\|. \tag{4.14}
\end{equation}

\medskip

This theorem, however, is incorrect. This was pointed out to us by Stefan Steinerberger who showed that
the inequality (4.14) does not hold for the function $f(\cos\t, \sin \t) = 1+ \varepsilon \sin \t$ for small enough
$\varepsilon$ when $d=2$. The mistake in the proof appeared on the line 6 of page 166, which states that $\|\tau(f)\| \le (2 |m_f| +1) \|g||_2^2$
but it should have been $\|\tau(f)\| \le \|g\|_2^2 + 2 |m_f| \|g\|_2$.  As a consequence, the right hand side of (4.14)
has to be replaced by $c_d \|\tau(f)\|^2$. Since $\|\tau(f) \| \le \|f\|_2^2$, the resulted inequality is then equivalent
to
\begin{equation}
      (1-\|\tau(f)\|^2) \|\nabla_0 f\|_{2}^2 \ge c_d \|\tau(f)\|^2, \tag{1}
\end{equation}
which was already known in the literature; see the discussion in \cite{DaiX14} and references therein.
%first established in \cite{NW} and studied by several other authors subsequently.

Since (4.14) no longer holds, an immediate question is whether the uncertainty principle in
(4.11) and that in (1) are equivalent, assuming $\int_{\sph} f(y)\, d\s(y)=0$. The following proposition shows that they are not 
equivalent and (4.11) is stronger than (1) for a large class of functions. 

\bigskip\noindent
{\bf Proposition 1.}
{\it For $n \ge 3$ let $Y \in \CH_n^d$, a  real spherical harmonic of degree $n$ on $\sph$, and let $Q$ be a real polynomial
of degree at most $n-2$ such that $\int_{\sph} Q(x) d\s =0$. Assume that both $[Y(x)]^2$ and $[Q(x)]^2$ are even
in every coordinate. Let
$$
       f = b (Y + Q), \quad\hbox{where}\quad b^{-1} := \|Y+Q\|_2 > 0.
$$
Then $\tau(f) = 0$. In particular, (1) becomes the trivial inequality
$\|\nabla_0 f\|_2^2 \ge 0$ whereas (4.11) shows that $\|\nabla_0 f\|_2^2 \ge c > 0$.
}
\medskip

\begin{proof}
Since the degree of $Q$ is at most $n-2$, it follows from the orthogonality of $Y$ and the even parity of
$Y^2$ and $Q^2$ that
$$
  \int_{\sph} x_i |f(x)|^2 d\s =  \int_{\sph} x_i \left(Y(x)^2 + 2 Y(x)Q(x) + Q(x)^2\right) d\s(x) =0
$$
for $1 \le  i \le d$. Hence, $\tau(f) =0$. By its definition, $\|f\|_2 =1$ and, by the orthogonality of $Y$ and
the zero mean of $Q$, we see that $\int_{\sph} f(x) d\s =0$ so that (4.11) is applicable to $f$.
\end{proof}

As a simple example of the function $f$, we can choose $Q(x) = x_1^k$ and $Y(x) = C_n^\l(x_1)$ for
$x =(x_1,\ldots, x_d) \in \sph$, where $\l = (d-2)/2$ and $1 \le k \le n -2$.

\bigskip
\noindent
{\bf Acknowledgement}. The authors thank Stefan Steinerberger for pointing out the mistake in \cite{DaiX14}.

\end{document}